\documentclass[12pt,a4paper,twoside]{amsart}

\usepackage{mathrsfs}          
\usepackage[all]{xy}    
\usepackage[latin1]{inputenc}
\usepackage{amscd}
\usepackage{latexsym}
\usepackage{multicol}
\usepackage{amsmath}
\usepackage{amssymb}
\usepackage{amsthm}

\theoremstyle{plain}
\newtheorem{Theorem}{Theorem}[section]
\newtheorem{lemma}[Theorem]{Lemma}
\newtheorem{corollary}[Theorem]{Corollary}
\newtheorem{proposition}[Theorem]{Proposition}

\newtheorem{conjecture2*}{Conjecture}
\newtheorem{remark2*}[conjecture2*]{Remark}
\newtheorem{Theorem2*}[conjecture2*]{Theorem}
\newtheorem{proposition2*}[conjecture2*]{Proposition}
\newtheorem{claim2*}[conjecture2*]{Claim}
\newtheorem{example2*}[conjecture2*]{Example}
\newtheorem{lemma2*}[conjecture2*]{Lemma}

\theoremstyle{definition}

\newtheorem{example}[Theorem]{Example}

\newtheorem{definition}[Theorem]{Definition}
\newtheorem{remark}[Theorem]{Remark}

\newtheorem{pkt}[Theorem]{}
\newtheorem{construction}[Theorem]{Construction}

\newcommand{\sH}{{\mathcal H}}

\newcommand{\sO}{{\mathcal O}}

\newcommand{\sS}{{\mathcal S}}

\newcommand{\sX}{{\mathcal X}}
\newcommand{\sY}{{\mathcal Y}}
\newcommand{\sZ}{{\mathcal Z}}

\newcommand{\B}{{\mathbb B}}
\newcommand{\C}{{\mathbb C}}

\newcommand{\E}{{\mathbb E}}
\newcommand{\F}{{\mathbb F}}

\newcommand{\bP}{{\mathbb P}}
\newcommand{\Q}{{\mathbb Q}}
\newcommand{\R}{{\mathbb R}}

\newcommand{\Z}{{\mathbb Z}}
\newcommand{\id}{{\rm id}}

\newcommand{\Hg}{{\rm Hg}}

\newcommand{\GL}{{\rm GL}}

\newcommand{\Sp}{{\rm Sp}}
\newcommand{\tr}{{\rm tr}}

\newcommand{\PU}{{\rm PU}}
\newcommand{\U}{{\rm U}}

\newcommand{\ad}{{\rm ad}}
\newcommand{\der}{{\rm der}}
\newcommand{\Gal}{{\rm Gal}}

\newcommand{\diag}{{\rm diag}}

\newcommand{\fX}{\mathfrak{X}}

\title{Maximal automorphisms of Calabi-Yau manifolds versus maximally unipotent monodromy}
\author{Jan Christian Rohde}
\address{GRK 1463 / Institut f\"ur Algebraische Geometrie\\ Leibniz Universit\"at Hannover\\ 30060 Hannover\\ Germany}
\email{rohde@math.uni-hannover.de}
\begin{document}
\maketitle

\begin{abstract}
Let $\alpha$ be an automorphism of the local universal deformation of a Calabi-Yau 3-manifold $X$, which does not act by $\pm \id$ on $H^3(X,\C)$. We show that the bundle $F^2(\sH^3)$ in the $VHS$ of each maximal family containing $X$ is constant in this case. Thus $X$ can not be a fiber of a maximal family with maximally unipotent monodromy, if such an automorphism $\alpha$ exists. Moreover we classify the possible actions of $\alpha$ on $H^3(X,\C)$, construct examples and show that the period domain is a complex ball containing a dense set of $CM$ points given by a Shimura datum in this case.
\end{abstract}

\section*{Introduction}

Due to their importance in theoretical physics, we are interested in Calabi-Yau 3-manifolds. We construct some examples of Calabi-Yau 3-manifolds $X$ with degree 3 automorphisms, which extend to the local universal deformation. Here such automorphisms are called maximal. Our examples of maximal automorphisms do not act by $\pm 1$ on $H^3(X,\C)$. The subbundle $F^2(\sH^3)$ of the variation of Hodge structures of the local universal deformation is constant, if a maximal automorphism exists and does not act by $\pm 1$ on $H^3(X,\C)$. Moreover we give an additional example of a Calabi-Yau 3-manifold, which does not necessarily have a maximal automorphism, but satisfies the condition that $F^2(\sH^3)$ is constant in the $VHS$ of the local universal deformation. We will see that $F^2(\sH^3)$ is constant for each maximal family containing $X$ as fiber, if this holds true with respect to the local universal deformation of $X$. This forbids $X$ to be a fiber of a maximal family  with maximally unipotent monodromy. Thus the assumptions of the formulation of the mirror conjecture in \cite{Mor} can not be satisfied by $X$, if $F^2(\sH^3)$ is constant in the local universal deformation of $X$.

Moreover we show that the period domain is a complex ball and the local universal deformation of $X$ has a dense set of complex multiplication ($CM$) fibers, if $X$ has a maximal automorphism, which does not act by $\pm 1$ on $H^3(X,\C)$. Theoretical physicists are interested in Calabi-Yau 3-manifolds with $CM$ - in particular if there exists a mirror pair of Calabi-Yau 3-manifolds with $CM$ (see \cite{GV}).

Moreover we will see that the quotient of the maximal automorphisms of $X$ by the automorphisms acting trivially on $H^3(X,\Z)$ is given by
$$\{e\}, \  \ \Z/(2), \  \ \Z/(3), \  \ \Z/(4) \   \ \mbox{or} \  \ \Z/(6).$$

\tableofcontents

\section{Examples with maximal automorphisms}

Here a Calabi-Yau $3$-manifold $X$ is a compact K\"ahler manifold of dimension $3$ such that 
$$\omega_X \cong \sO_X \  \ \mbox{and} \  \ h^{k,0}(X) = 0 \  \ \mbox{for} \  \ k = 1,2.$$
Let $\sX \to B$ be the local universal deformation of $X = \sX_0$. We say that a family $f:\sY \to \sZ$ of Calabi-Yau 3-manifolds is maximal, if for each $z \in \sZ$ there exists an open neighborhood $U$ of $z$ such that $\sY_U$ is isomorphic to the Kuranishi family of $\sY_z$. Recall that
$$H^3 := R^3f_*(\Q)$$
is a local system and that
$$\sH^3 := H^3\otimes_{\Q}\sO_{\sZ}$$
is a holomorphic bundle. The variation of Hodge structures of weight 3 is given by the filtration
$$0 \subset F^3(\sH^3) \subset F^2(\sH^3) \subset F^1(\sH^3)\subset \sH^3$$
by holomorphic subbundles.

Recall that a marked $K3$ surface is a pair $(S,\mu)$ consisting of a $K3$ surface $S$ and a marking $\mu$, that is an isometry $\mu:L\to H^2(S,\Z)$ of lattices, where
$$L = U\oplus U\oplus U\oplus -E_8\oplus -E_8.$$
The marked $K3$ surfaces $(S,\mu)$ and $(S',\mu')$ are isomorphic, if there exists an isomorphism $f:S \to S'$ such that $\mu =  f^*\circ\mu'$. By gluing marked local universal deformations of $K3$ surfaces, we obtain the complex analytic moduli space $M$ of marked $K3$ surfaces with universal family $f:\sS \to M$. Moreover let $\phi$ denote an isometry of order $3$ on $L$ and $(L_{\C})_{\eta}$ denote the eigenspace on $L_{\C}$ with eigenvalue $\eta$ with respect to $\phi$. For this section we fix
$$\xi = \exp(\frac{2\pi i}{3}) \   \  \mbox{and} \  \ r = \dim (L_{\C})_{\xi} -1.$$

\begin{construction}
Let $\mu_m: L\to H^2(\sS_m,\Z)$ denote a marking on the fiber $\sS_m$ of $\sS$ such that the isomorphism class of $(\sS_m,\mu_m)$ is represented by $m \in M$. One has the new marking $\mu_m\circ \phi$ on each fiber $\sS_m$. By the universal property of the universal family this yields the holomorphic maps $\alpha_S:\sS \to \sS$ and $\alpha_M:M \to M$ such that
$$\xymatrix{
{\sS} \ar[rr]^{\alpha_{\sS}} \ar[d]^{f} &  & {\sS} \ar[d]^{f}\\
{M} \ar[rr]^{\alpha_{M}} &   & {M}
}$$
is a commutative diagram of holomorphic maps. Let $\Delta_M\subset M\times M$ denote the diagonal,
$$M_{\alpha} = {\rm Graph}(\alpha_M) \cap \Delta_M$$
and assume that there exists an $m\in M_{\alpha}$ such that $\alpha|_{\sS_m}$ is a non-symplectic automorphism. The space $M_{\alpha}$ is complex analytic, but not necessarily Hausdorff. We obtain the restricted family $\sS_{\alpha}\to M_{\alpha}$ with a non-symplectic $M_{\alpha}$-automorphism of degree 3. Since the pullback action of $\alpha_{\sS}|_{\sS_m}$ on $H^2(\sS_m,\Z)$ is given by $\mu_m\circ\phi\circ\mu_m^{-1}$, one has $H^{2,0}(\sS_m)\subset \mu_m(L_{\C})^{\bot}_1$. Assume without loss of generality that $H^{2,0}(\sS_m)\subset \mu_m(L_{\C})_{\xi^{-1}}$. The intersection form yields a Hermitian form on $(L_{\C})_{\xi^{-1}}$ of signature $(1,r)$, since for each positive definite eigenvector $v \in (L_{\C})_{\xi^{-1}}$ one has also a positive definite eigenvector $\bar v \in (L_{\C})_{\xi}$ and the Hermitian form on $L_{\C}$ has signature $(3,19)$. Thus the possible choices for $\mu_m^{-1}(H^{2,0}(\sS_m))$ are given by the points of the corresponding ball $\B_r\subset \bP((L_{\C})_{\xi^{-1}})$. The period map of $\sS_{\alpha}$ is a locally injective multivalued map to $\B_r$, since the local universal deformation of each $K3$ surface has an injective period map. Since each point of $\B_r$ yields an up to a scalar unique vector $\omega$ with
$$\omega\cdot \omega = 0 \ \ \mbox{and} \  \ \omega\cdot \bar \omega > 0,$$
this ball is contained in the period domain of $K3$ surfaces. Thus the period map $M_{\alpha}\to \B_r$ is locally bijective.
\end{construction}

\begin{Theorem} \label{huhu}
For $0 \leq k \leq 6$ there exists a non-symplectic automorphism of degree 3 of a $K3$ surface with a fixed locus consisting of $k$ disjoint smooth rational curves and $k + 3$ isolated fixed points. One has that
$$r = 6-k.$$
\end{Theorem}
\begin{proof}
(see \cite{AS}, Figure 1 and \cite{T}, Table 2)
\end{proof}

\begin{remark}
In \cite{CH}, Theorem $3.3$ an example with $k = 6$ has been constructed. An example of a family with $k = 3$ occurs in \cite{DK}, Section 12 and \cite{JCR}, Chapter 8. In the appendix of \cite{JCR} one finds an explicitly constructed fiber for $k = 1$.
\end{remark}

\begin{construction}
Let $\E$ denote the Fermat curve of degree 3 and $\alpha_{\E}$ denote a degree 3 automorphism of $\E$ acting via pullback by $\xi$ on $H^{1,0}(\E)$. The quotient $\sS_{\alpha}\times \E/\langle(\alpha_{\sS}, \alpha_{\E})\rangle$ is birationally equivalent to a family $\fX_{\alpha}\to M_{\alpha}$ of Calabi-Yau 3-manifolds. This construction method has been studied in \cite{CH}, Proposition $3.1$ and \cite{JCR}, Section $9.2$.

From their actions on $H^{2,0}(S)$ and $H^{1,0}(\E)$ one concludes that $\alpha_{\E}$ acts by $\xi$ and $\alpha_{\sS}$ acts by $(\xi^2,1,\ldots,1)$ and $(\xi,\xi,1,\ldots,1)$ near their respective fixed loci. The singularities of $\sS_{\alpha}\times \E_3/\langle(\alpha_{\sS}, \alpha_{\E_3})\rangle$ consist of families of curves and sections over $M_{\alpha}$. The singular sections can be resolved by blowing up the fixed sections over them before the application of the quotient map. In the case of the singularities given by families of curves the action of $\langle(\alpha_{\sS}, \alpha_{\E})\rangle$ is given by $(\xi,\xi^2,1,\ldots, 1)$ near the corresponding fixed families of curves. We blow up the fixed families of curves with exceptional divisor $E_1$ and in a second step we blow up the families of fixed curves contained in $E_1$ with exceptional divisor $E_2$. Let $$\widetilde{\sS_{\alpha}\times \E_3}\to M_{\alpha}$$
denote the family obtained from all previous blowing up transformations. We have the quotient map
$$\psi: \widetilde{\sS_{\alpha}\times \E_3} \to \tilde \fX _{\alpha}:=\widetilde{\sS_{\alpha}\times \E_3}/\langle(\alpha_{\sS}, \alpha_{\E})\rangle$$
such that $\tilde \fX_{\alpha}$ is smooth. By blowing down $\psi(\tilde E_1)$ to a family of curves, we obtain a crepant resolution.
\end{construction}

\begin{proposition} \label{zuzu}
Assume that the codimension one fixed locus in $\sS_{\alpha}$ with respect to $\alpha_{\sS}$ consists of families of rational curves. Then our family $\fX_{\alpha} \to M_{\alpha}$ is maximal.
\end{proposition}
\begin{proof}
Note that
$$H^{3,0}(\fX_m) = H^{2,0}(\sS_m) \otimes H^{1,0}(\E)$$
for each $m \in M_{\alpha}$. Moreover the period map of the family $\sS_{\alpha}$ is a locally bijective multivalued map to the ball $\B_r$. Since $\E$ is a fixed curve with fixed Hodge structure, the period map of $\sS_{\alpha}$ yields the period map of $\fX_{\alpha}$. Thus the period map of $\fX_{\alpha}$ is a locally bijective multivalued map to the ball $\B_r$. 
By the fact that we only blow up and down $\bP^1$-bundles over families of rational curves and $\bP^2$-bundles over sections of $\sS_{\alpha}\times \E_3\to M_{\alpha}$, our birational transformations do not have any effect on the third cohomology of the fibers (follows from \cite{Voi}, Th\'eor\`eme $7.31$). Since
$$b_1(\sS_m) = b_3(\sS_m) =0,$$
one concludes that
$$H^3(\fX_m,\C)  \  \ = \  \  H^3(\sS_m\times \E,\C)_1$$
$$ = \ \ H^{1,0}(\E)\otimes(H^{2,0}(\sS_m) \oplus H^{1,1}(\sS_m)_{\xi^2}) \  \ \oplus \  \
H^{0,1}(\E)\otimes(H^{0,2}(\sS_m) \oplus H^{1,1}(\sS_m)_{\xi}).$$
Thus
$$h^{2,1}(\fX) = h^{1,1}(\sS_m)_{\xi^2} = r,$$
which implies that $\fX_{\alpha}$ is maximal.
\end{proof}

\begin{proposition}
Our family $\fX_{\alpha} \to M_{\alpha}$ has a degree 3 automorphism over its basis acting by $\xi$ on $F^2(\sH^3)$.
\end{proposition}
\begin{proof}
The automorphism $\alpha_{\E}$ acts on the family $\sS_{\alpha}\times \E$ and commutes with $(\alpha_{\sS},\alpha_{\E})$.  On the exceptional divisor over the isolated fixed sections $(\alpha_{\sS},\alpha_{\E})$ acts trivially. One can easily check that the actions of $\alpha_{\E}$ and $(\alpha_{\sS},\alpha_{\E})$ are inverse to each other on $\tilde E_1$. Moreover $(\alpha_{\sS},\alpha_{\E})$ acts trivially on $E_2$. Thus the automorphism $\alpha_{\E}$ acts also on $\widetilde{\sS_{\alpha}\times \E}$ and commutes with the induced action of $(\alpha_{\sS},\alpha_{\E})$. Due to the fact that $\alpha_{\E}$ fixes $\tilde E_1$, it descends to an automorphism of $\fX$. By the proof of Proposition $\ref{zuzu}$, we have
$$F^2(H^3(\fX_m,\C)) = H^{1,0}(\E)\otimes(H^{2,0}(\sS_m) \oplus H^{1,1}(\sS_m)_{\xi^2}).$$
Since $\alpha_{\E}$ acts by $\xi$ on $H^{1,0}(\E)$ and the action of $\alpha_{\E}$ on $F^2(H^3(\fX_m,\C))$ is given by the action on the corresponding differential forms on $\sS_{m}\times \E$ via pullback, $\alpha_{\E}$ acts by  $\xi$ on $F^2(\sH^3)$.
\end{proof}

\begin{proposition} \label{sc}
The fibers of $\fX_{\alpha} \to M_{\alpha}$ are simply connected.
\end{proposition}
\begin{proof}
Let $C$ be a complex manifold. The fundamental group of $\bP^N$ is trivial. Hence
$$\pi_1(\tilde C) = \pi_1(C)$$
for each blowing up $\tilde C$ of a smooth curve or a point. Note that $K3$ surfaces are simply connected. Thus $\pi_1(\sS_p \times\E)$ is given by $H_1(\E,\Z)= H^1(\E,\Z)^*$. By the following Lemma and the action of $(\alpha_{\sS},\alpha_{\E})$ on $\pi_1(\sS_p \times\E)\cong H_1(\E,\Z)$ given by the the well-known action of $\alpha_{\E}$ on $H_1(\E,\Z)$, one obtains the stated result.
\end{proof}

\begin{lemma} \label{FundGr}
Let $f:X \to Y$ be a cyclic covering of manifolds of degree $n$, whose Galois group fixes at least one point $p\in X$ and $$\{\gamma\cdot (g_*(\gamma))^{-1}|\gamma \in \pi_1(X,p),g \in {\rm Gal}(f)\}=\pi_1(X,p).$$ Then
$$\pi_1(Y) = 0.$$
\end{lemma}
\begin{proof}
One can lift each path $\gamma$ on $Y$ with $\gamma(0) = \gamma(1) = f(p)$ to $n$ closed paths on $X$ with starting and ending point $p$. Each of these paths is mapped to $\gamma$. Thus the induced homomorphism $f_*:\pi_1(X,p)\to \pi_1(Y,f(p))$ is surjective.

Since for each $g\in \Gal(f)$ and $\gamma \in \pi_1(X,p)$ one has
$$f_*(\gamma) = f_*g_*(\gamma),$$
one concludes from the assumptions that $f_*$ is the zero map.
\end{proof}

\begin{Theorem} \label{gagu}
For $0 \leq k \leq 6$ one has a maximal family of simply connected Calabi-Yau 3 manifolds $X$ with a maximal automorphism acting by $\xi$ on $F^2(H^3(X,\C))$. The Hodge numbers are given by
$$h^{2,1}(X) = 6-k \  \  \mbox{and} \  \ h^{1,1}(X)  = 18 + 11\cdot k.$$
\end{Theorem}
\begin{proof}
By the previous results, it remains only to compute $h^{1,1}(X)$. Since each $K3$ surface $S$ has the Betti numbers $b_1(S) = 0$ and $b_2(S) = 22$, one concludes that the eigenspace $h^{1,1}(\sS_m\times \E)_1$ of $h^{1,1}(\sS_m\times \E)$ with the eigenvalue 1 with respect to $(\alpha_{\sS},\alpha_{\E})$ is given by:
$$h^{1,1}(\sS_m\times \E)_1 = h^{0,0}(\sS_m)\cdot h^{1,1}(\E) + h^{1,1}(\sS_m)_1\cdot h^{0,0}(\E)$$
$$= 22 - 2(r+1) + 1 = 21 - 2r$$
For each isolated fixed section on $\sS_{\alpha}$ we have to blow up 3 sections on $\sS_{\alpha}\times \E$. Moreover for each fixed family of curves on $\sS_{\alpha}$ we have altogether to blow up 9 families of curves and to blow down 3 families of curves. By Theorem $\ref{huhu}$, we have $0 \leq k \leq 6$ fixed families of curves, $n=k+3$ isolated fixed sections on $\sS_{\alpha}$ and $r = 6-k$. Thus we conclude:
$$h^{1,1}(X) = h^{1,1}(\sS_m\times \E)_1 + 6\cdot k + 3 \cdot n = 21 - 2r+ 6\cdot k + 3\cdot (k+3)$$
$$= 21-12+2\cdot k+ 6\cdot k + 3\cdot k + 9 = 18 + 11\cdot k$$
\end{proof}

\section{Constant $F^2$-bundle}

In this section we show that there are some Calabi-Yau 3-manifolds, which cannot occur as fibers of a maximal family with maximally unipotent monodromy. This comes about, if $F^2(\sH^3)$ is constant, a concept which we will make precise here

Let $f:\sY \to \sZ$ be a maximal holomorphic family of Calabi-Yau 3-manifolds with $X \cong \sY_z$, where $\sZ$ is an arcwise connected topological space covered by open charts of subsets of $\C^{h^{2,1}(X)}$ such that the gluing maps are biholomorphic.\footnote{Apart from the facts that $\sZ$ does not need to be Hausdorff and the topology of $\sZ$ does not need to have a countable basis, $\sZ$ can be considered as a manifold. The author has made these assumtions to make also sure that there is not a pathological basis which allows maximally unipotent monodromy.} Moreover assume that $z \in U \subset \sZ$ is an open and contractible manifold and consider the period map
$$p_{U}:U \to {\rm Grass}(H^3(X,\C),b_3(X)/2)$$
associating to each $u \in U$ the subspace
$$F^2(H^3(\sY_u,\C))\subset H^3(\sY_u,\C) \cong H^3(\sY_U,\C) \cong H^3(X,\C)$$
as described in \cite{Voi}, Chapter 10. We say that $F^2(\sH^3)$ is constant over $U$, if the period map $p_{U}:U \to {\rm Grass}(H^3(X,\C),b_3(X)/2)$ is constant.
Moreover let $\gamma:[0,1] \to \sZ$ be a closed path with
$$\gamma(0) = \gamma(1) = z.$$
Let $t \in [0,1]$ and $\gamma^{-1}(H^3_{\C})$ denote the inverse image sheaf of $H_{\C}^3= H^3 \otimes \C$. Since we have a canonical isomorphism between $H_{\C}^3$ and the locally constant sheaf associated to the presheaf
$$V \to H^3(f^{-1}(V),\C|_{f^{-1}(V)}),$$
we obtain a natural isomorphism
$$\gamma(t)^*: H^3(\sY_{\gamma(t)},\C)\to(H^3_{\C})_{\gamma(t)} \to \gamma^{-1}(H^3_{\C})_t.$$

\begin{definition} \label{defit}
The bundle $F^2(\sH^3)$ is constant along $\gamma$, if $\gamma([0,1])$ can be covered by open contractible manifolds $U_1,\ldots,U_N$ such that the period maps $p_{U_1},\ldots,p_{U_N}$ are constant.
\end{definition}

\begin{remark}\label{broa}
Assume that $F^2(\sH^3)$ is constant along $\gamma$. That means that $\gamma([0,1])$ can be covered by open contractible sets $U_1,\ldots,U_N$ such that $F^2(\sH^3)$ is constant over each $U_i$. Let
$$V_1\cup\ldots\cup V_{N'} = [0,1]$$
be a finite subcovering of the covering of connected components of the several $\gamma^{-1}(U_i)$ and $t_i \in V_i$. On each $V_i$ we define a period map
$$p_{V_i}:V_i\to {\rm Grass}(\gamma^{-1}(H^3_{\C})_{t_i},b_3(X)/2)$$
associating to each $t \in V_i$ the subspace
$$\gamma(t)^*(F^2(H^3(\sY_{\gamma(t)},\C)))\subset \gamma^{-1}(H^3_{\C})_{t} \cong \gamma^{-1}(H^3_{\C})(V_i) \cong \gamma^{-1}(H^3_{\C})_{t_i}.$$
The assumption that $F^2(\sH^3)$ is constant over each $U_i$ implies that each $p_{V_i}$ is constant. Since $[0,1]$ is simply connected,  $\gamma^{-1}(H^3_{\C})$ is a constant sheaf. Therefore the constant maps $p_{V_i}$ can be glued to a constant map
$$\gamma^*p:[0,1] \to {\rm Grass}(\gamma^{-1}(H^3_{\C})_1,b_3(X)/2)$$
associating to each $t \in [0,1]$ the subspace
$$\gamma(t)^*(F^2(H^3(\sY_{\gamma(t)},\C)))\subset \gamma^{-1}(H^3_{\C})_t \cong \gamma^{-1}(H^3_{\C})([0,1]) \cong \gamma^{-1}(H^3_{\C})_1.$$
\end{remark}

\begin{lemma} \label{hehe}
Assume that $F^2(\sH^3)$ is constant along $\gamma \in \pi_1(\sZ,z)$. Then one has
$$\rho(\gamma)(F^2(H^3(X,\C))) = F^2(H^3(X,\C)).$$
\end{lemma}
\begin{proof}
Recall that the monodromy of $H^3_{\C}$ is defined by
$$\rho(\gamma) = (\gamma(1)^*)^{-1}\circ\eta\circ\gamma(0)^*,$$
where $\eta:\gamma^{-1}(H^3_{\C})_0\to \gamma^{-1}(H^3_{\C})_1$ is the natural isomorphism
$$\gamma^{-1}(H^3_{\C})_0 \cong \gamma^{-1}(H^3_{\C})([0,1]) \cong \gamma^{-1}(H^3_{\C})_1.$$
Since we assume that $F^2(\sH^3)$ is constant along $\gamma$, the map $\gamma^*p$ is constant by Remark $\ref{broa}$. In other terms we have
$$\eta((\gamma(0)^*)(F^2(H^3(X,\C)))) = \gamma(1)^*(F^2(H^3(X,\C))).$$
Thus we conclude the stated result from the definition of monodromy.
\end{proof}

\begin{lemma} \label{lem2}
Let $U_1,\ldots,U_N$ be a finite covering of $\gamma([0,1])$ by open contractible sets. Then $F^2(\sH^3)$ is constant along $\gamma$, if $F^2(\sH^3)$ is constant over $U_i$ for some $i$.
\end{lemma}
\begin{proof}
By our assumption, we have without loss of generality that $p_{U_1},\ldots,p_{U_k}$ are constant for $1\leq k$. If $k < N$, one has
$$V_k \cap \tilde V_k \neq \emptyset,$$
$$\mbox{where} \ \ V_k := U_1\cup \ldots\cup U_{k} \ \ \mbox{and} \ \ \tilde V_k = U_{k+1}\cup \ldots\cup U_{N}.$$
Without loss of generality  there is a $U_{k+1}$ with $V_k \cap U_{k+1}\neq \emptyset$. Note that on a connected complex manifold $M$ a holomorphic map is constant, if it is constant on an open subset of $M$. Hence $p_{U_{k+1}}$ is constant, since $p_{U_{k+1}}|_{V_k \cap U_{k+1}}$ is locally constant. By the fact that
$$S = \{i \in \{1,\ldots,N\}|p_{U_i} \ \ \mbox{constant}\}$$
satisfies $1 \leq\sharp S \leq N$ and there cannot be a $0<k < N$ with $\sharp S = k$, we get $\sharp S = N$.
\end{proof}

\begin{Theorem} \label{eigen}
Assume that the bundle $F^2(\sH^3)$  is constant in the $VHS$ of the Kuranishi family of $X$. Then each matrix of the monodromy representation of $H^3=R^3f_*(\Q)$ is given by $$\left(\begin{array}{cc}
M & 0 \\
0  &  \bar M
\end{array} \right)
$$
for some $\frac{b_3(X)}{2} \times \frac{b_3(X)}{2}$ matrix $M$ acting on $F^2(H^3(X,\C))$ and $\bar M$ acting on $\overline{F^2(H^3(X,\C))}$.
\end{Theorem}
\begin{proof}
The bundle $F^2(\sH^3)$ is constant along $\gamma$ for all $\gamma \in \pi_1(\sZ,z)$ in the sense of Definition $\ref{defit}$, since there exists an open neighborhood $U'$ of $z$ such that $p_{U'}$ is constant by the assumptions (see Lemma $\ref{lem2}$). By Lemma $\ref{hehe}$, this implies that $F^2(H^3(X,\C))$ is fixed by the monodromy.

The monodromy representation of $H^3$ is given by a homomorphism
$$\rho:\pi_1(\sZ,z) \to  \GL(H^3(X,\Q)).$$
Thus for each $v \in F^2(H^3(X,\C))$ one has
$$\rho(\gamma)(\bar v) = \overline{\rho(\gamma)(v)}.$$
Therefore each matrix of the monodromy representation of $H^3=R^3f_*(\Q)$ fixes both $F^2(H^3(X,\C))$ and $\overline{F^2(H^3(X,\C))}$. Since
$$H^3(X,\C) = F^2(H^3(X,\C)) \oplus \overline{F^2(H^3(X,\C))},$$
one obtains the stated result.
\end{proof}

\begin{pkt}
In the literature \cite{CK}, \cite{Gr} one finds a formulation of the mirror conjecture, which was given in \cite{Mor}, Section 7 first. For this formulation of the mirror conjecture one needs maximally unipotent monodromy defined in \cite{Mor}, Definition 3. In the case of maximally unipotent monodromy the monodromy representation yields some unipotent matrices $T_1, \ldots, T_k$. One chooses
the matrix $N$ by a linear combination of logarithms of $T_1, \ldots, T_k$ and obtains a weight filtration
$$0 \subset W_0 \subseteq W_1  \subseteq \ldots \subseteq W_6 = H^3(X,\Q)\  \ \mbox{with} \  \ W_0 = {\rm Im }(N^3) \  \ \mbox{and} \   \ \dim W_0 = 1$$
in the case of maximally unipotent monodromy.

Now assume that we are in the situation of Theorem $\ref{eigen}$. In this case $N^3$ acts by $M$ on $F^2(H^3(X,\C))$ and by $\bar M$ on $\overline{F^2(H^3(X,\C))}$. Hence ${\rm Im }(N^3)$ has an even dimension. But 1 is not an even number. Therefore in the case of Theorem $\ref{eigen}$ the assumptions of this formulation of mirror symmetry can not be satisfied.
\end{pkt}

\begin{example} \label{ex1}
Let $S$ be a $K3$ surface with involution $\iota_S$ acting by 1 on $H^{1,1}(S)$ and by $-1$ on $H^{2,0}(S) \oplus H^{0,2}(S)$. By taking an elliptic curve $E$ with involution $\iota_E$ and the Borcea-Voisin construction \cite{Bc2}, \cite{Voi2} obtained from blowing up the singularities of
$$S \times E/\langle(\iota_S,\iota_E)\rangle,$$
one gets a Calabi-Yau 3-manifold $X$, whose Kuranishi family has a constant bundle $F^2(\sH^3)$. The Calabi-Yau 3-manifold $X$ has the Hodge numbers
$$h^{1,1}(X) = 61 \  \ \mbox {and} \  \ h^{2,1}(X) = 1$$
(for details see \cite{JCR}, Example $1.6.9$ and \cite{JCR}, Example $11.3.11$).\footnote{The fact that $F^2(\sH^3)$ is constant follows from the arguments that the $VHS$ of the Kuranishi family of $X$ depends only on the $VHS$ of the corresponding deformation of elliptic curves and
$$F^2(X) = H^{3,0}(X)\oplus H^{2,1}(X) = H^{2,0}(S) \otimes H^{1,0}(E) \oplus H^{2,0}(S) \otimes H^{0,1}(E) =  H^{2,0}(S) \otimes H^{1}(E,\C).$$}
\end{example}

\begin{remark} \label{haha}
The pullback action of an extension of a maximal automorphism $\alpha$ of $X$ to the local universal deformation $\sX \to B$ yields an induced action on the constant sheaf $H^3_B$ and hence on $\sH^3_B = H^3_B\otimes \sO_B$. Each eigenspace ${\rm Eig}(\alpha,\eta)\subset \sH^3_B$ is parallel with respect to the Gauss-Manin connection, that means $\nabla_{\chi}s\in {\rm Eig}(\alpha,\eta)(U)$ for each $\chi \in TB(U)$ and $s\in {\rm Eig}(\alpha,\eta)(U)$, where $U \subset B$ open. There does not exist a nowhere vanishing holomorphic section $s\in\sH^3_B(U)$ such that each $s_{b}\in {\rm Eig}(\alpha,\eta_b)_b$ for some eigenvalue $\eta_b$ and $\eta_{b_1} \neq \eta_{b_2}$ for some $b_1,b_2 \in U$. Since for each $b \in B$ the space $F^3(\sH^3)_b$ has to be contained in some eigenspace, there is a fixed eigenvalue $\eta$ such that $F^3(\sH^3)_{B}\subset {\rm Eig}(\alpha,\eta)$. Let the tangent space $T_bB$ be generated by the basis
$$\frac{\partial}{\partial x_1}, \ldots, \frac{\partial}{\partial x_{h^{2,1}(X)}}.$$
For a holomorphic section $\omega$ of $F^3(\sH^3)_B$ with $\omega(b) \neq 0$ one has
$$\frac{\partial \omega}{\partial x_1}(b), \ldots, \frac{\partial \omega}{\partial x_{h^{2,1}(X)}}(b)\in {\rm Eig}(\alpha,\eta)_b.$$
Since
$$\omega(b),\frac{\partial \omega}{\partial x_1}(b), \ldots, \frac{\partial \omega}{\partial x_{h^{2,1}(X)}}(b)$$
is a basis of $F^2(\sH^3)_b$ and our observations hold true for each $b \in B$, one concludes that
$$F^2(\sH^3)_B \subset {\rm Eig}(\alpha,\eta).$$
Now assume that $\eta$ is not real. One concludes similarly to the proof of Theorem $\ref{eigen}$ that
$$F^2(\sH^3)_B = {\rm Eig}(\alpha,\eta) \ \ \mbox{and} \ \ \overline{F^2(\sH^3)_B} ={\rm Eig}(\alpha,\bar \eta),$$
since $\alpha$ is defined over $\Q$ and
$$H^3(X,\C) = F^2(H^3(X,\C)) \oplus \overline{F^2(H^3(X,\C))}.$$
Thus the period map
$$p_B:B\to {\rm Grass}(H^3(X,\C),b_3(X)/2)$$
is constant and the assumptions of Theorem $\ref{eigen}$ are satisfied. Therefore $X$ cannot be a fiber of a maximal family with maximally unipotent monodromy, if $X$ has a maximal automorphism acting by a non-real eigenvalue on $H^{3,0}(X)$. In particular this holds true for the examples constructed in Section 1 (see Theorem $\ref{gagu}$).
\end{remark}

\section{Consequences for complex multiplication}

In this section we explain and prove the following theorem:

\begin{Theorem}\label{gugu}
Assume that $X$ has a maximal automorphism $\alpha$, which acts by a non-real eigenvalue $\eta$ on $H^{3,0}(X)$. Then the Kuranishi family $\sX \to B$ of $X$ has a dense set of $CM$ fibers.
\end{Theorem}

We follow the conventions and notations of \cite{JCR}. Let $V$ be a $\Q$-vector space of finite dimension,
$$S^1 = {\rm Spec}(\R[x,y]/(x^2+y^2-1)) \  \ \mbox{with} \  \ S^1(\R) \cong \{z \in \C|z\bar z = 1\}$$
and $h: S^1 \to \GL(V_{\R})$ be a homomorphism of $\R$-algebraic groups. Each rational Hodge structure of some fixed weight $k$ is given by a pair $(V,h)$. The Hodge group $\Hg(V,h)$ is the smallest $\Q$-algebraic subgroup $G$ of $\GL(V_{\Q})$ such that $h(S^1) \subset G_{\R}$. Recall that a Calabi-Yau 3-manifold $X$ has $CM$, if $\Hg(H^3(X,\Q),h_X)$ is a torus algebraic group. For a characterization of Calabi-Yau 3-manifolds with $CM$ recall the following facts:

\begin{pkt} \label{jac}
The Calabi-Yau 3-manifold $X$ has the following intermediate Jacobians (see \cite{Bc}):
\begin{itemize}
\item The Griffiths intermediate Jacobian\index{intermediate Jacobian!Griffiths} ${\rm J}_G(X)$ is the complex torus corresponding to the Hodge structure of type $(1,0),(0,1)$ on $H^3(X,\Z)$ given by
$$H^{1,0} := H^{3,0}(X) \oplus H^{2,1}(X),  \  \ H^{0,1} := H^{1,2}(X) \oplus H^{0,3}(X) .$$
\item The Weil intermediate Jacobian \index{intermediate Jacobian!Weil}${\rm J}_W(X)$ is the abelian variety corresponding to the Hodge structure of type $(1,0),(0,1)$ on $H^3(X,\Z)$ given by
$$H^{1,0} := H^{2,1}(X) \oplus H^{0,3}(X),  \  \ H^{0,1} := H^{3,0}(X) \oplus H^{1,2}(X).$$
\end{itemize}
\end{pkt}

Calabi-Yau 3-manifolds with $CM$ are characterized by the following proposition:

\begin{proposition}
A Calabi-Yau 3-manifold $X$ has $CM$, if and only if the Hodge groups of the weight one Hodge structures corresponding to ${\rm J}_G(X)$ and ${\rm J}_W(X)$ are tori and commute.\footnote{Some authors write that a Calabi-Yau 3-manifold has $CM$, if and only if its Griffiths intermediate Jacobian has $CM$. For a proof they incorrectly quote the same article \cite{Bc}.}
\end{proposition}
\begin{proof}
(see \cite{Bc}, Theorem $2.3$)
\end{proof}

The proof of Theorem $\ref{gugu}$ follows arguments and ideas similar to \cite{JCR}, Section $4.3$ and Section $4.4$. We use the theory of Shimura varieties, which is explicated in \cite{De2}, \cite{Dat}. The vector space automorphism $J:H^3(X,\C) \to H^3(X,\C)$ acting by $i$ on $F^2(X)$ and $-i$ on $\overline{ F^2(X)}$ satisfies $J(H^3(X,\R)) = H^3(X,\R)$. Thus $J$ is a complex structure on $H^3(X,\R)$, that means $J^2 = -\id$. Let $(H^3(X,\R),J)$ denote the resulting complex vector space. By
$$F^2(H^3(X,\C))\to (H^3(X,\R),J), \   \  v \to \tilde v = v +\bar v,$$
we have an isomorphism of complex vector spaces. The cup product yields an alternating form $Q$ on $H^3(X,\Q)$. By $Q$, we get the Hermitian form
\begin{equation} \label{defi}
H(\cdot,\cdot) = iQ(\cdot,\bar \cdot)
\end{equation}
on $H^3(X,\C)$ such that the Hodge decomposition is orthogonal with respect to $H$. As we have seen in Remark $\ref{haha}$ the maximal automorphism $\alpha$ acts by $\eta$ on $F^2(\sH^3)_B$ and by $\bar \eta$ on $\overline{F^2(\sH^3)_B}$, if the eigenvalue $\eta$ of the action of $\alpha$ on $H^{3,0}(X)$ is not real. Moreover let
$$C(\alpha)\subset \Sp(H^3(X,\Q),Q)$$
denote the centralizer of the action of $\alpha$ on $H^3(X,\Q)$.
Let for each $N \in \GL(F^2(H^3(X,\C)))$ the automorphism $\bar N \in \GL(\overline{F^2(H^3(X,\C))})$ be given by
$$\bar N \bar v = \overline{Nv} \  \  \   \ (\forall v \in F^2(H^3(X,\Q))).$$
We obtain the isomorphism $tr: \GL(F^2(H^3(X,\C))) \to \GL(H^3(X,\R),J)$ of $\C$-algebraic groups by
$$\GL(F^2(H^3(X,\C)))\ni N \to (N, \bar N) \in \GL(F^2(H^3(X,\C)))\times \GL(\overline{F^2(H^3(X,\C))}).$$

\begin{lemma}
$$C(\alpha)(\R) = \tr(\U(F^2(H^3(X,\C)),H|_{F^2(H^3(X,\C))})(\R))$$
\end{lemma}
\begin{proof}
Assume that $N \in \U(F^2(H^3(X,\C)),H|_{F^2(H^3(X,\C))})(\R)$. Then $tr(N)$ fixes the two eigenspaces $F^2(H^3(X,\C))$ and $\overline{F^2(H^3(X,\C))}$. Thus it commutes with $\alpha$. By $\eqref{defi}$, each
$$N \in \U(F^2(H^3(X,\C)),H|_{F^2(H^3(X,\C))})(\R)$$
satisfies that $\tr(N) \in \Sp(H^3(X,\Q),Q)(\R)$.
Thus
$$\tr(\U(F^2(H^3(X,\C)),H|_{F^2(H^3(X,\C))})(\R)) \subseteq C(\alpha)(\R).$$

Let $M \in C(\alpha)(\R)$. Thus $M$ fixes the two eigenspaces $F^2(H^3(X,\C))$ and $\overline{F^2(H^3(X,\C))}$. Since $M \in \Sp(H^3(X,\Q),Q)(\R)$, one obtains $M \in \tr(\U(F^2(H^3(X,\C)),H|_{F^2(H^3(X,\C))})(\R))$ by using $\eqref{defi}$.
Thus
$$C(\alpha)(\R) \subseteq \tr(\U(F^2(H^3(X,\C)),H|_{F^2(H^3(X,\C))})(\R)).$$
\end{proof}

Since $\alpha$ yields a Hodge isometry of $(H^3(X,\Q),h_X)$, one obtains
$$h(S^1) \subset C(\alpha)_{\R}.$$
The adjoint representation $h^{\ad}$ yields the homomorphism
$$S^1(\R) \to C(\alpha)^{\ad}(\R) \cong \PU(1,h^{2,1}(X))(\R) \ \ \mbox{given by} \  \ z \to [\diag(z,z^{-1},\ldots,z^{-1})]$$
and the centralizer $K$ of $h(S^1)$ in $C(\alpha)(\R)$ is isomorphic to $(\U(1) \times \U(h^{2,1}(X)))(\R)$. Thus $(C(\alpha)^{\ad},h^{\ad})$ is a Shimura datum, which yields the complex ball $\B_{h^{2,1}(X)}$.

Due to the fact that $F^2(\sH^3)_B$ is constant (follows from Remark $\ref{haha}$), the period map of the Kuranishi family $\sX \to B$ is given by the fractional period map
$$p_{F^2}:B \to \bP(F^2(H^3(X,\C))).$$
Since the period map of the Kuranishi family is injective (see \cite{Voi3}, Lemma $1.5$) and
$$h^{2,1}(X) = \dim \bP(F^2(H^3(X,\C))) = \dim B,$$
$p_{F^2}$ is open. Moreover each open subset of the period domain
$$C(\alpha)^{\ad}(\R)/\ad(K)\cong \B_{h^{2,1}(X)}$$
associated to the Shimura datum $(C(\alpha)^{\ad},h^{\ad})$ is given by an open set of $\bP(F^2(H^3(X,\C)))$. Thus we have an open map
$$B \to C(\alpha)^{\ad}(\R)/\ad(K)=C(\alpha)(\R)/K,$$
which assigns to each point $b\in B$ the Hodge structure on $H^3(\sX_b,\Q)$.

Let $\Hg(h^{\ad})$ be the smallest $\Q$-algebraic subgroup of $C(\alpha)^{\ad}$ with $h^{\ad}(S^1) \subset \Hg(h^{\ad})_{\R}$. Note that $\Hg(H^3(X,\Q),h_X)\subset \ad^{-1}(\Hg(h^{\ad}))$ and $\Hg(h^{\ad})$ is a torus, if and only if the derived group $(\ad^{-1}(\Hg(h^{\ad})))^{\der}$ is contained in the kernel of $\ad$ given by the center $Z(C(\alpha))$. Hence if $\Hg(h^{\ad})$ is a torus, $$\Hg^{\der}(H^3(X,\Q),h_X)\subseteq (\ad^{-1}(\Hg(h^{\ad})))^{\der}\subseteq Z(C(\alpha)).$$
Due to the fact that a reductive group is the almost direct product of its derived group and its center, this implies that the reductive group $\Hg(H^3(X,\Q),h_X)$ is a torus, if $\Hg(h^{\ad})$ is a torus. Since the set of points $h^{\ad}\in C(\alpha)^{\ad}(\R)/\ad(K)$ such that $\Hg(h^{\ad})$ is a torus is dense (follows from \cite{JCR}, Theorem $1.7.2$), we obtain Theorem $\ref{gugu}$.

\section{The cohomology actions of maximal automorphisms}

Let $X$ be a Calabi-Yau 3-manifold. In order to obtain all possible actions of maximal automorphism on $H^3(X,\C)$ we have to start with some general observations about Hodge structures on Calabi-Yau 3-manifolds and their endomorphism algebras. It is a well-known fact that one can decompose rational Hodge structures into simple rational Hodge structures, that are indecomposable rational Hodge structures. Let $(H_S^3(X,\Q),h_S)$ denote the simple rational sub-Hodge structure of $(H^3(X,\Q),h_X)$ which satisfies
$$H^{3,0}(X) \subset H_S^3(X,\C).$$

\begin{lemma} \label{NrField}
Sending $\gamma \in{\rm End}_{\Q}(H_S^3(X,\Q),h_S)$ to the eigenvalue of the action of $\gamma$ on $H^{3,0}(X)$ yields an isomorphism
$${\rm End}_{\Q}(H_S^3(X,\Q),h_S) \to \F \subset \C$$
of rings, where $\F$ is a number field.
\end{lemma}
\begin{proof}
(see \cite{Bc})
\end{proof}

\begin{lemma} \label{wurzel}
Each automorphism $\alpha$ of $X$ acts on $H^{3,0}(X)$ by the multiplication with a root of unity.
\end{lemma}
\begin{proof}
By the assumptions, $\dim H^{3,0}(X) = 1$ and $\alpha(H^{3,0}(X)) = H^{3,0}(X)$. Thus $\alpha$ acts by the multiplication with an eigenvalue $\eta$ on $H^{3,0}(X)$. Since $\alpha$ fixes $H^3(X,\R)$, it acts by $\bar \eta$ on $H^{0,3}(X)$. Moreover $\alpha$ respects the polarization on $X$. Thus one has also that $\alpha$ acts by $\eta^{-1}$ on $H^{0,3}(X)$. Hence
$$\eta^{-1} = \bar \eta \Leftrightarrow |\eta| = 1.$$
Since the action of $\alpha$ on $H_S^3(X,\Q)$ is $\Q$-rational and commutes with $h_S(S^1)$, the action of $\alpha$ yields an element of ${\rm End}_{\Q}(H_S^3(X,\Q),h_S)$. By Lemma $\ref{NrField}$, this implies that $\eta$ is contained in a number field $\F$. Thus $\eta$ is a root of unity.
\end{proof}

\begin{pkt} \label{nn}
Now consider a not necessarily maximal automorphism $\alpha$ of $X$, which does not act trivially on $H^3(X,\C)$ and satisfies that $\alpha^p$ acts trivially on $H^3(X,\C)$ for some prime number $p> 2$. The minimal polynomial of the action of $\alpha$ on $H^3(X,\Q)$ divides
$$x^p-1 = (x-1)(x^{p-1}+ \ldots x +1).$$
Thus
$$H^3(X,\Q) = {\rm Eig}_{\Q}(\alpha,1) \oplus N^3(X,\Q),$$
where $N^3(X,\Q)$ is denotes the subspace of $H^3(X,\Q)$, on which $\alpha$ acts as an automorphism with minimal polynomial $x^{p-1}+ \ldots x +1$.
Let $\eta$ be a primitive $p$-th. root of unity and consider
$$N^3(X,\Q(\eta)) = N^3(X,\Q)\otimes \Q(\eta).$$
The action of $\alpha$ yields a decomposition of $N^3(X,\Q(\eta))$ into the eigenspaces ${\rm Eig}_{\Q(\eta)}(\alpha,\eta^r)$ with $r = 1, \ldots, p-1$. Let $\gamma \in \Gal(\Q(\eta),\Q)$. Since the action of $\alpha$ on $H^3(X,\Q(\eta))$ is defined by rational matrices, $\alpha$
and $\gamma$ commute. Hence for $v \in {\rm Eig}_{\Q(\eta)}(\alpha,\eta^r)$ one has
$$\gamma(\eta^r)\gamma(v) = \gamma(\eta^r v) = (\gamma\circ \alpha)(v) = \alpha(\gamma(v)).$$
Thus all eigenspaces ${\rm Eig}_{\Q(\eta)}(\alpha,\gamma(\eta^r))$ have the same dimension $d$. Since there are $p-1$ primitive $p$-th. roots of unity, one has
$$\dim N^3(X,\Q) = d \cdot (p-1).$$
\end{pkt}

\begin{Theorem} \label{hoho}
Assume that a maximal automorphism $\alpha$ of $X$ does not act by $\pm 1$ on $H^{3}(X,\C)$. Then there exists a root of unity $\eta\neq \pm 1$ such that
$$F^2(H^3(X,\C)) = {\rm Eig}(\alpha,\eta) \ \ \mbox{and} \ \ \overline{F^2(H^3(X,\C))} ={\rm Eig}(\alpha,\bar \eta),$$
$F^2(\sH^3)_B$ is constant, the Kuranishi family has a dense set of $CM$ fibers and $X$ cannot occur as fiber of a maximal family with maximally unipotent monodromy. 
\end{Theorem}
\begin{proof}
By Lemma $\ref{wurzel}$, the eigenvalue $\eta$ of the action of the maximal automorphism $\alpha$ on $H^{3,0}(X)$ is a root of unity. In Remark $\ref{haha}$ we have seen that $\alpha$ acts by $\eta$ on $F^2(H^3(X,\C))$. Moreover we have that $\alpha$ acts by $\bar \eta$ on $\overline{F^2(H^3(X,\C))}$, since the action of $\alpha$ is defined over $\Q$. By the assumption that $\alpha$ does not act by $\pm 1$ on $H^{3}(X,\C)$, one concludes $\eta\neq \pm 1$. Hence the root of unity $\eta$ is not real and the rest of the theorem follows from the discussion in Remark $\ref{haha}$ and Theorem $\ref{gugu}$.
\end{proof}

\begin{lemma}
Assume that the maximal automorphism $\alpha$ does not act by $\pm 1$ on $H^3(X,\C)$ and $\alpha^p$ acts trivially on $H^3(X,\C)$ for some prime number $p$. Then $p = 3$.
\end{lemma}
\begin{proof}
Since $\alpha$ does not act by $\pm 1$ on $H^{3}(X,\C)$, we conclude from Theorem $\ref{hoho}$ that
$$d = \dim F^2(H^3(X,\C)) = \frac{b_3(X)}{2} \  \ \mbox{and} \  \ {\rm Eig}_{\Q}(\alpha,1) = 0.$$
Thus:
$$\frac{b_3(X)}{2}= d = \frac{b_3(X)}{p-1} \Rightarrow p = 3$$
\end{proof}

From Theorem $\ref{hoho}$ one concludes that each maximal automorphism $\alpha$ has a smallest positive integer $m$ such that $\alpha^m$ acts trivially on $H^3(X,\C)$. Hence:

\begin{corollary}
Assume that the maximal automorphism $\alpha$ does not act trivially on $H^3(X,\C)$. Then there exists an $m$, which has only the prime divisors 2 and $3$, such that $\alpha^m$ acts trivially on $H^3(X,\C)$.
\end{corollary}

Assume that 9 is the smallest positive integer $m$ such that $\alpha^m$ acts trivially on $H^3(X,\C)$. Then the minimal polynomial divides
$$x^9-1 = (x^6+x^3+1)(x^2+x+1)(x-1)$$
and there exists a subspace $N^3(X,\Q)\subset H^3(X,\Q)$ such that the restriction of the action of $\alpha$ to $N^3(X,\Q)$ has the minimal polynomial $x^6+x^3+1$. By the same arguments as in $\ref{nn}$, the vectorspace $N^3(X,\Q)\otimes \C$ decomposes into 6 eigenspaces with the same dimension with respect to the 6 primitive 9-th.  roots unity. Since the action of a maximal automorphism on $H^3(X,\C)$ yields atmost two eigenspaces, there does not exist a maximal automorphism $\alpha$ with $m=9$.
Moreover there are 4  primitive 8-th. roots of unity and 4 primitive 12-th. roots of unity. Thus we conclude:

\begin{Theorem}
The group of maximal automorphisms of $X$ is up to the subgroup of automorphisms acting trivially on $H^3(X,\C)$ given by
$$\{e\}, \  \ \Z/(2), \  \ \Z/(3), \  \ \Z/(4) \   \ \mbox{or} \  \ \Z/(6).$$
\end{Theorem}

\section*{Acknowledgments}
This paper was written at the Graduiertenkolleg ``Analysis, Geometry and Stringtheory'' at Leibniz Universit\"at Hannover. I would like to thank Klaus Hulek very much for many stimulating discussions and his hint to the article \cite{CH}, which yields the construction method of the examples in Section 1. Moreover I would like to thank Lars Halle and Derek Harland for discussing the arguments in Section 2.

\end{document}